\documentclass[11pt]{amsart}
\usepackage{geometry}                
\usepackage{graphicx}
\usepackage{amssymb}
\usepackage{epstopdf}

\usepackage[colorlinks=true,pdfpagemode=none,urlcolor=blue, linkcolor=blue,citecolor=blue]{hyperref}
\usepackage{amssymb,amsmath,amsfonts,amssymb}
\usepackage{graphics,graphicx,color,verbatim}
\usepackage{color}

\def\@abssec#1{\vspace{.05in}\footnotesize \parindent .2in
{\bf #1. }\ignorespaces}

\graphicspath{{/EPSF/}{Figures/}}
\DeclareGraphicsExtensions{.eps}

\newtheorem{theorem}{Theorem}[section]

\newtheorem{remark}[theorem]{Remark}

\def \Rm {\mathbb R}

\def \Cm {\mathbb C}

\def \Sm {\mathbb S}
\newcommand{\eps}{\varepsilon}

\newcommand{\dsum}{\displaystyle\sum}

\newcommand{\mF}{\mathcal F}
\newcommand{\mG}{\mathcal G}
\newcommand{\mH}{\mathcal H}

\newcommand{\mK}{\mathcal K}

\newcommand{\mP}{\mathcal P}

\newcommand{\mX}{\mathcal X}
\newcommand{\mY}{\mathcal Y}

\newcommand{\mJ}{{\mathfrak J}}

\newcommand{\cout}[1]{}

\title{Local Inversions in Ultrasound Modulated Optical Tomography}

\author{Guillaume Bal}
\address{Department of Applied Physics and Applied Mathematics, Columbia University, New York, NY 10027}
\email{gb2030@columbia.edu}
\thanks{G.B. was supported in part by  NSF grant DMS-1108608 and AFOSR Grant NSSEFF- FA9550-10-1-0194. S.M. was supported in part by NSF grant DMS-1108858.}

\author{Shari Moskow}
\address{Department of Mathematics, Drexel University}
\email{moskow@math.drexel.edu}

\begin{document}
\maketitle

\begin{abstract}
Ultrasound modulated optical tomography, also called acousto-optics tomography, is a hybrid imaging modality that aims to combine the high contrast of optical waves with the high resolution of ultrasound. We follow the model of the influence of ultrasound modulation on the light intensity measurements developed in \cite{BS-PRL-10}. We present sufficient conditions ensuring that the absorption and diffusion coefficients modeling light propagation can locally be uniquely and stably reconstructed from the corresponding available information. We present an iterative procedure to solve such a problem based on the analysis of linear elliptic systems of redundant partial differential equations.
\end{abstract}

\section{Introduction}

This paper concerns the reconstruction of the diffusion and absorption coefficients $\gamma$ and $\sigma$ for a model of light propagation in tissues given by  
\begin{equation}\begin{array}{ll}
\label{eq:diff0}-\nabla\cdot \gamma  \nabla u + \sigma u=0 \quad &\mbox{ in } X\\
u=f & \mbox{ on } \partial X,
\end{array}
\end{equation}
where $u(x)$ is the light intensity of diffuse photons propagating in a bounded open domain $X\subset\Rm^n$ for $n\geq2$ with Dirichlet boundary condition given by $f$ on (the sufficiently smooth) $\partial X$, the boundary of $X$. More precise models of boundary conditions can be used with no consequence on the results presented in this paper and so we present our result with Dirichlet conditions to simplify notation.

The reconstruction of $(\gamma,\sigma)$ from optical boundary measurements is known to be a severely ill-posed problem in many settings of light propagation and results in reconstructions with very poor resolution \cite{AS-IP-10,B-IP-09}. Ultrasound Modulated Optical Tomography (UMOT), also called Acousto-Optics Tomography (AOT), aims to combine the high contrast of optical coefficients observed in many diseases with the high resolution of ultrasound. We refer to \cite{AFRBG-OL-05,KLZZ-JOSA-97,W-JDM-04} for additional information on this hybrid modality. In this paper, we follow the model of ultrasound modulation of light intensities developed in \cite{BS-PRL-10}. In this model, the difference of light measurements with and without ultrasound modulation provides, to a first approximation, internal functionals of the unknown parameters of the form
\begin{equation}
\label{eq:H}
H(x) = \gamma(x) |\nabla u|^2(x) + \eta \sigma(x) u^2(x),\qquad x\in X,
\end{equation}
where $\eta$ is a known constant in the model. The internal functional $H(x)$ is parameterized by the boundary condition $f$ on $\partial X$. The objective of UMOT is to reconstruct the unknown parameters $(\gamma,\sigma)$ from knowledge of a minimum of functionals $H_j$ for $1\leq j\leq J$ corresponding to {\em well-chosen} boundary conditions $\{f_j\}_{1\leq j\leq J}$.

Many studies are devoted to the problem with $\sigma\equiv0$ above. This problem, which finds applications in ultrasound modulated electrical impedance tomography (UMEIT), also called acousto-electric tomography or impedance-acoustic tomography, is now well understood and we refer the reader to \cite{ABCTF-SIAP-08,B-UMEIT-12,BBMT-12,BNSS-JIIP-13,CFGK-SJIS-09,GS-SIAP-09,KK-AET-11,MB-aniso-13,MB-IP-12,MB-IPI-12} for a detailed account of such theories. Note that \cite{MB-aniso-13,MB-IP-12} analyze the reconstruction of tensor-valued (anisotropic) coefficients $\gamma$, which we do not consider in this paper for UMOT. The reconstruction of two coefficients with more measurements than proposed above is considered in \cite{ACDRT-SIAP-11}.

Pioneered by the work in \cite{CFGK-SJIS-09}, several papers \cite{B-UMEIT-12,BBMT-12,MB-aniso-13,MB-IP-12,MB-IPI-12} provide explicit reconstruction procedures for $\gamma$ (when $\sigma\equiv0$) in the setting of a highly redundant measurements corresponding to a large value of $J$. Such procedures do not extend to the reconstruction of $(\gamma,\sigma)$ from knowledge of $H_j$ in \eqref{eq:H} for $1\leq j\leq J$. Rather, we follow a standard method to solve the above nonlinear problem consisting of analyzing its linearization and introducing a standard fixed point iterative procedure provided the linearized operator is injective. Such a strategy was followed in the setting $\sigma\equiv0$ in, e.g., \cite{B-Irvine-12,KK-AET-11, KS-IP-12}. The papers \cite{B-Irvine-12,KS-IP-12} provide general strategies to solve similar problems. In this paper, we follow the method developed in \cite{B-Irvine-12} that recasts the UMOT inverse problem as an elliptic redundant system of partial differential equations for which we can apply the classical elliptic regularity results developed in \cite{ADN-CPAM-64,DN-CPAM-55,S-JSM-73}.

More precisely, we recast the UMOT inverse problem as a system of partial differential equations in section \ref{sec:model} and present some sufficient conditions so that the system be {\em elliptic}. Ellipticity provides an inversion procedure up to the possible existence of a finite dimensional kernel. This is not sufficient, and conditions for injectivity are given here in section \ref{sec:injectivity}. Once the linearized UMOT inverse problem is injective, standard fixed point iteration procedures recalled in section \ref{sec:nonlinear} may be applied to solve the nonlinear problem locally. The conditions for ellipticity and injectivity are shown to be satisfied for appropriate choices of the boundary conditions $\{f_j\}$ in section \ref{sec:cond}. The practically and pedagogically interesting case of constant backgrounds is presented in section \ref{sec:constant}.

The UMOT inverse problem is an example of a large class of hybrid inverse problems (also called coupled-physics or multi-wave inverse problems) that aim to combine one modality with high contrast with another modality with high resolution. For an incomplete list of books and reviews on this active field of research, we refer the reader to, e.g., \cite{A-Sp-08,AS-IP-12,B-IO-12,PS-IP-07,S-SP-2011}.

\section{System and ellipticity}
\label{sec:model}
 We consider the equation
 \begin{equation}\begin{array}{ll}
\label{eq:diff}-\nabla\cdot \gamma  \nabla u_j + \sigma u_j =0 \quad &\mbox{ in } X\\
u_j=f_j & \mbox{ on } \partial X,
\end{array}
\end{equation}
for $1\leq j\leq J$ with $J\geq2$ and availability of functionals of the form
\begin{equation}
\label{eq:Hj}
\gamma|\nabla u_j|^2 + \eta \sigma |u_j|^2 = H_{j}(x)\quad \mbox{ in } X , \qquad 1\leq j\leq J,
\end{equation}
for some fixed, known, constant $0\not=\eta\in\Rm$. The equations \eqref{eq:diff}-\eqref{eq:Hj} may be seen as a redundant $2J\times (2+J)$ system of nonlinear partial differential equations for the dependent variables $(\gamma,\sigma,\{u_j\}_{1\leq j\leq J})$. As indicated above, in the setting $\sigma\equiv0$, explicit reconstruction procedures for $(\gamma,\{u_j\})$ have been obtained in \cite{B-UMEIT-12,BBMT-12,CFGK-SJIS-09,MB-aniso-13,MB-IP-12,MB-IPI-12} when $J$ is sufficiently large. The extension of such methods to the reconstruction of  $(\gamma,\sigma)$ directly from knowledge of $\{H_j\}_{1\leq j\leq J}$ is not known at present. Moreover, as shown in \cite{B-Irvine-12} in the setting $\sigma\equiv0$, $\gamma$ can uniquely and stably be reconstructed from $\{H_j\}_{1\leq j\leq J}$ for a smaller value of $J$ than what is necessary in the available explicit reconstruction procedures. 

The method followed in \cite{B-Irvine-12}, which shares many similarities with that described in \cite{KS-IP-12} (the main difference being that \cite{KS-IP-12} considers the inversion of systems of pseudo-differential operators while \cite{B-Irvine-12} considers the inversion of larger, but differential, systems of operators), first consists of linearizing the nonlinear problem \eqref{eq:diff}-\eqref{eq:Hj}. Replacing
$\gamma\to\gamma+\delta\gamma$, $\sigma\to\sigma+\delta\sigma$, and $u_j\to u_j+\delta u_j$ as well as $H_j\to H_j+\delta H_j$ and considering the terms in \eqref{eq:Hj}-\eqref{eq:diff} that are linear in ``$\delta$", we calculate that
\begin{eqnarray}
\label{eq:dHj}
  |\nabla u_j|^2 \delta \gamma+ \eta|u_j|^2 \delta\sigma + 2\gamma \nabla u_j \cdot\nabla\delta u_j 
+2\eta\sigma_0  u_j \delta u_j&=&\delta H_{j} ,\\[2mm]
\label{eq:deltau}
  L\delta u_j - \nabla\cdot\delta\gamma \nabla u_j +  u_j \delta \sigma&=& 0 ,\qquad L := -\nabla\cdot\gamma\nabla + \sigma.
\end{eqnarray}
Moreover, we obtain that $\delta u_j=0$ on $\partial X$.
Let us introduce the notation
\begin{displaymath}
   F_j = \nabla u_j,\qquad\theta_j=\dfrac{\nabla u_j}{|\nabla u_j|},\qquad d_j=\dfrac{u_j}{|\nabla u_j|},
\end{displaymath}
and the operator
\begin{equation}
\label{eq:Mj}
  M_j v = 2\gamma F_j \cdot\nabla v+2\eta\sigma  u_j v.
\end{equation}
Then, applying $L$ to \eqref{eq:dHj}, $M_j$ to \eqref{eq:deltau}, and introducing the commutator $[L,M_j]=LM_j-M_jL$, we find that 
\begin{eqnarray}
\label{eq:dHj2}
  \Big(L |F_j|^2 \delta \gamma+M_j \nabla\cdot \delta\gamma F_j\Big) + \Big(L \eta|u_j|^2 \delta\sigma-M_j u_j\delta\sigma \Big) + [L,M_j] \delta u_j &=& L \delta H_{j} ,\\[2mm]
\label{eq:deltau2}
  L\delta u_j - \nabla\cdot\delta\gamma \nabla u_j +  u_j \delta \sigma&=& 0.
\end{eqnarray}
This is a linear system of equations for $v=(\delta\gamma,\delta\sigma,\{\delta u_j\}_{1\leq j\leq J})$, which may be recast as
\begin{equation}
\label{eq:PQJ}
A_J v= (P_J + Q_J) v = S_J,
\end{equation}
with $P_J$ a $2J\times (2+J)$ matrix of second-order operators given by  
\begin{equation}
\label{eq:PJ}
P_J = \left(\begin{matrix} \vdots & \vdots & \vdots & \vdots & \vdots\\-\gamma |F_j|^2 \Delta + 2\gamma F_j\otimes F_j : \nabla\otimes \nabla & \eta |u_j|^2 \Delta & \ldots & [L,M_j] & \ldots \\  0 & 0 & \ldots &-\gamma\Delta &\ldots \\ \vdots & \vdots & \vdots& \vdots & \vdots  \end{matrix}\right)
\end{equation}
and $Q_J$ a $2J\times (2+J)$ matrix of at most first-order differential operators, whose explicit expression we do not reproduce. The source $S_J$ is a $2J\times 1$ matrix with odd entries $2j-1$ given by $L\delta H_j$ and vanishing even entries. Here $a\otimes a$ for an $n-$vector $a$ is the rank-one matrix of components $a_ia_j$; and for two matrices $A$ and $B$, we denote $A:B={\rm Tr} (A^*B)$ the usual inner product with $A^*$, the Hermitian conjugate to $A$.

The lower-order term $Q_J$ has a complicated structure, which we hope not to analyze in detail. We thus look for conditions (on the boundary conditions $f_j$ and on the number of measurements $J$ ) which guarantee that $P_J$ is an elliptic operator. Let $p_J(x,\xi)$ be the (principal) symbol of $P_J$. The operator $P_J$ is said to be elliptic when for each $\xi\in \Sm^{n-1}$, the $2J\times (2+J)$ matrix $p_J(x,\xi)$ has full rank $2+J$.  The latter matrix is given by
\begin{equation}
\label{eq:pJ}
p_J(x,\xi) = \left(\begin{matrix} \vdots & \vdots & \vdots & \vdots & \vdots\\\gamma(x) |F_j|^2(x) |\xi|^2 - 2\gamma(x) \big(F_j(x)\cdot\xi\big)^2 & -\eta |u_j|^2(x) |\xi|^2 & \ldots & \check p_j(x,\xi) & \ldots \\  0 & 0 & \ldots &\gamma(x)|\xi|^2 &\ldots \\ \vdots & \vdots & \vdots& \vdots & \vdots  \end{matrix}\right),
\end{equation}
with $\check p_j(x,\xi)$ a quadratic form in $\xi$ whose explicit expression does not influence ellipticity since $\gamma(x)|\xi|^2$ is uniformly bounded from below by a constant times $|\xi|^2$, which equals $1$ for $\xi\in\Sm^{n-1}$. Therefore, $p_J(x,\xi)$ has full rank for a given $\xi\in\Sm^{n-1}$ if and only if the following $J\times 2$ matrix has full rank (i.e. rank equal to $2$):
\begin{equation}
\label{eq:hpJ}
\hat p_J(x,\xi) = \left(\begin{matrix} \vdots & \vdots\\ \gamma(x) \Big(|F_j|^2(x) |\xi|^2 - 2\big(F_j(x)\cdot\xi\big)^2\Big) & -\eta |u_j|^2(x) |\xi|^2 \\\vdots & \vdots \end{matrix}\right).
\end{equation}
We assume, for instance by imposing that $f_j$ is uniformly bounded from below by a positive constant, that $|u_j|$ is also 
 uniformly bounded from below by a positive constant. As a consequence, we observe that the above matrix is full rank if and only if the following matrix is full rank
 \begin{equation}
\label{eq:tpJ}
\tilde p_J(x,\xi) = \left(\begin{matrix} \vdots & \vdots\\ \Big(|\xi|^2 - 2(\theta_j\cdot\xi)^2\Big) & - |d_j|^2 |\xi|^2 \\\vdots & \vdots \end{matrix}\right),
\end{equation}
with the dependence in $x$ dropped to simplify notation. We define the quadratic (in $\xi$) forms
\begin{equation}
\label{eq:pj}
p_j(x,\xi) = |\xi|^2 - 2\theta_j(x)\otimes \theta_j(x) : \xi\otimes \xi= |\xi|^2 - 2(\theta_j(x)\cdot\xi)^2.
\end{equation}
Were we to be in the situation where the absorption coefficient $\sigma$ vanishes in the above equations, then $\tilde p_J$ is full rank (i.e. of rank $1$) provided that $\{p_j(x,\xi)=0,\quad 1\leq j\leq J\}$ implies that $\xi=0$. In other words, the light cones with direction $\theta_j$ intersect only at the point $\xi=0$; see \cite{B-Irvine-12} for an analysis of this problem in the linearized setting.

Here, in the general $\sigma$ case, we obtain that the matrix $\tilde p_J(x,\xi)$ is of rank two for all $\xi\in \Sm^{n-1}$ if and only if the quadratic forms
\begin{equation}
\label{eq:pjk}
p_{jk}(x,\xi) = |d_j|^2(x) p_k(x,\xi) - |d_k|^2(x) p_j(x,\xi),
\end{equation}
are such that 
\begin{equation}
\label{eq:ellcond}
\big\{p_{jk}(x,\xi)=0,\,\, 1\leq j<k\leq J\big\} \,\,\mbox{ implies that } \,\,\xi=0.
\end{equation}

Let us assume that the ellipticity condition \eqref{eq:ellcond} holds. This provides uniform ellipticity inside the open domain $X$. We now need to find boundary conditions that preserve the elliptic structure of $P_J$. Such conditions are called the Lopatinskii conditions, and under the assumption \eqref{eq:ellcond}, it is not difficult to show that they are satisfied for Dirichlet boundary conditions for $v$; see \cite{B-Irvine-12}. 

Thus, using the theories of elliptic systems described in \cite{ADN-CPAM-64,S-JSM-73} (see also \cite{B-Irvine-12}), we obtain results that we collect in the following theorem:
\begin{theorem}
\label{thm:ellipticity}
 Consider the system \eqref{eq:PQJ} augmented with boundary conditions $v=0$ on $\partial X$. We assume that $P_J$ is elliptic in the sense that \eqref{eq:ellcond} holds. Then the system admits a left parametrix $B_J$ such that
\begin{equation}
\label{eq:BAJ}
 v = B_J S_J + T_J v,
\end{equation}
where $T_J$ is a compact operator. More precisely, $T_J$ is a pseudo-differential operator of order $-1$ (of class $L^{-1}_{1,0}$ as defined in, e.g., \cite{GS-CUP-94}). The above expression may be extended by a parametrix that solves the above system up to a smoothing operator of arbitrary order. That is, we have
\begin{equation}\label{eq:BAJ2}
v =  \dsum_{q=0}^{m-1} T_J^q B_J S_J  + T_J^m v
\end{equation}
for  any $m$ where $T_J^m$ is a pseudo-differential operator of order $-m$, mapping functions (or distributions) in $H^s(X)$ to functions in $H^{s+m}(X)$.
Moreover, there exist constants $C$ and $C_2$ such that 
\begin{equation}
\label{eq:stabest2}
  \|(\delta\gamma,\delta\sigma)\|_{H^{s}(X;\Rm^{2})} + \|\delta u_j\|_{H^{s+1}(X;\Rm^{J})} \leq C \|\delta H_j\|_{H^{s}(X;\Rm^J)} +C_2 \|v\|_{L^2(X;\Rm^{2+J})}.
\end{equation} 
This estimate is optimal since $\|\delta H_j\|_{H^{s}(X;\Rm^J)}$ is bounded by a constant times the above right-hand side.
\end{theorem}
\begin{proof}
Elliptic theory in \cite{ADN-CPAM-64,S-JSM-73} provides a left parametrix $B_J$ such that
\begin{equation}
\label{eq:BAJ3}
 B_J(P_J + Q_J) v  = (I-T_J) v = B_J S_J,
\end{equation}
where $T_J$ is a compact operator, and more precisely a pseudo-differential operator of order $-1$ (of class $L^{-1}_{1,0}$ as defined in, e.g., \cite{GS-CUP-94} when the coefficients are assumed to be sufficiently smooth). When the coefficients in the differential operators are sufficiently smooth, the perturbation $T_J$ may be replaced by an arbitrarily smoothing operator.  That is, for any $m$ we have
\begin{displaymath}
  (I- T_J^m) v = \dsum_{q=0}^{m-1} T_J^q B_J S_J,
\end{displaymath}
where $T_J^m$ a pseudo-differential operator of order $-m$, bounded from $H^s(X)$ to $H^{s+m}(X)$.
Moreover, standard elliptic regularity results \cite{ADN-CPAM-64,S-JSM-73} provide the following stability estimates
\begin{equation}
\label{eq:stabest1}
  \|v\|_{H^{s+2}(X;\Rm^{2+J})} \leq C \|S_J\|_{H^{s}(X;\Rm^{2J})} \leq C \|\delta H_j\|_{H^{s+2}(X;\Rm^J)} +C_2 \|v\|_{L^2(X;\Rm^{2+J})}.
\end{equation} 
These estimates are not optimal as the source term in \eqref{eq:deltau2} vanishes. From standard elliptic regularity for the latter equation, we obtain that $\delta u_j$ is one derivative smoother than $\delta \gamma$ and hence we obtain \eqref{eq:stabest2}. That the latter estimate is optimal is clear from \eqref{eq:dHj}.
\end{proof}

\begin{remark}
\label{rem:directestimate} \rm
The estimate \eqref{eq:stabest2} can be obtained directly from \eqref{eq:dHj2}-\eqref{eq:deltau2} by introducing an elliptic system in the sense of Douglis and Nirenberg \cite{DN-CPAM-55}.  Assigning the weights $s_{2j-1}=0$ to \eqref{eq:dHj2} and $s_{2j}=-1$ to \eqref{eq:deltau2}, and the weights $t_1=t_2=2$ for the columns for $(\delta\gamma,\delta\sigma)$ and the weights $t_j=3$ for $3\leq j\leq 2+J$ for the columns for $\{\delta u_j\}$, we find that the leading term in \eqref{eq:dHj2}-\eqref{eq:deltau2} is
\begin{equation}
\label{eq:tPJ}
\tilde P_J = \left(\begin{matrix} \vdots & \vdots & \vdots & \vdots & \vdots\\-\gamma |F_j|^2 \Delta + 2\gamma F_j\otimes F_j : \nabla\otimes \nabla & \eta |u_j|^2 \Delta & \ldots & 0 & \ldots \\  -F_j\cdot\nabla  & 0 & \ldots &-\gamma\Delta &\ldots \\ \vdots & \vdots & \vdots& \vdots & \vdots  \end{matrix}\right).
\end{equation}
This operator is elliptic if and only if $P_J$ is elliptic and standard weighted elliptic regularity estimates \cite{ADN-CPAM-64,S-JSM-73}  then directly provide \eqref{eq:stabest2}; we refer to \cite{B-Irvine-12} for additional details, which we do not reproduce here.
\end{remark}
\begin{remark}
\label{rem:directestimate2} \rm
The same estimates may in fact be established at a lower cost in terms of the number of redundant boundary conditions. Indeed, consider the problem \eqref{eq:dHj}-\eqref{eq:deltau} directly. Then we can also write this system in the sense of Douglis and Nirenberg assigning the weights $s_{2j-1}=0$ to \eqref{eq:dHj} and $s_{2j}=-1$ to \eqref{eq:deltau}, as well as the weights $t_1=t_2=0$ for the columns for $(\delta\gamma,\delta\sigma)$ and the weights $t_j=1$ for $3\leq j\leq 2+J$ for the columns for $\{\delta u_j\}$. For such weights, we would find that the leading term is of the form
\begin{equation}
\label{eq:tPJ}
{\mP}_J = \left(\begin{matrix} \vdots & \vdots & \vdots & \vdots & \vdots\\ |F_j|^2   & \eta |u_j|^2 & \ldots & 2\gamma F_j\cdot\nabla  & \ldots \\  -F_j\cdot\nabla  & 0 & \ldots &-\gamma\Delta &\ldots \\ \vdots & \vdots & \vdots& \vdots & \vdots  \end{matrix}\right).
\end{equation}
It is not difficult to verify that the symbol of this operator (replacing $\nabla$ by $i\xi$ and $\Delta$ by $-|\xi|^2$) is injective if and only if $P_J$ above is elliptic. Moreover, we verify as in the case $\sigma\equiv0$ in \cite{B-Irvine-12} that under the same conditions of ellipticity, then the Lopatinskii conditions are satisfied provided that $\delta u_j=0$; with no conditions necessary for $(\delta\sigma,\delta\gamma)$ on $\partial X$. Elliptic regularity then again provides a result of the form \eqref{eq:stabest2} with the last term replaced by $C_2\|\delta u_j\|_{L^2(X;\Rm^{J})}$. In some sense, this latter result is more favorable and shows that the application of $L$ to \eqref{eq:dHj} above requires that we introduce additional boundary conditions for $(\delta\sigma,\delta\gamma)$ on $\partial X$. Since the injectivity results of the following section can only be established for systems of the form \eqref{eq:dHj2}-\eqref{eq:deltau2}, we decided to directly present the latter in detail. Nonetheless, the following more precise version of Theorem \ref{thm:ellipticity} holds:
\begin{theorem}
\label{thm:ellipticity2}
Consider the system  \eqref{eq:dHj}-\eqref{eq:deltau}  with Dirichlet conditions $\delta u_j=0$ on $\partial X$. Then under the same ellipticity condition \eqref{eq:ellcond} as in Theorem \ref{thm:ellipticity}, the existence of parametrices for $v$ solution of \eqref{eq:dHj}-\eqref{eq:deltau} as described in \eqref{eq:BAJ} and \eqref{eq:BAJ2} still holds. Moreover, we have the (optimal) stability estimate
\begin{equation}
\label{eq:stabest3}
  \|(\delta\gamma,\delta\sigma)\|_{H^{s}(X;\Rm^{2})} + \|\delta u_j\|_{H^{s+1}(X;\Rm^{J})} \leq C \|\delta H_j\|_{H^{s}(X;\Rm^J)} +C_2 \|\delta u_j\|_{L^2(X;\Rm^{J})}.
\end{equation} 
\end{theorem}
\end{remark}

\section{Generic injectivity}
\label{sec:injectivity}
The presence of the constant $C_2$ in the preceding estimates comes from the fact that the operator $P_J+Q_J$ augmented with Dirichlet conditions need not be injective. This is reminiscent of the possible lack of invertibility of $-\Delta +V$ for a given potential $V$. Elliptic theory provides a parametrix $\Delta_D^{-1}$ (inversion of the Laplace operator with Dirichlet boundary conditions) so that $-\Delta +V$ is replaced by $I-\Delta_D^{-1} V$, with $\Delta_D^{-1} V$ a compact operator as $T_J$ above. However, we cannot guarantee that $1$ is not an eigenvalue of $T_J$ in general unless we have a precise understanding of the lower-order term $Q_J$. 

We present here a general methodology to ensure generic injectivity of an operator of the form $P_J+Q_J$ at the cost of imposing additional boundary conditions. Again, consider the scalar problem $Av :=(-\Delta +V)v=0$ with prescribed Dirichlet conditions on $\partial X$. Let us replace the above system by $A^*Av=0$, which is a fourth-order system, with its {\em own} Dirichlet conditions, which in this case would correspond to prescribing $v$ and $\partial_\nu v$ on $\partial X$. By integrations by parts, we obtain that $A^*Av=0$ along with $v=\partial_\nu v=0$ on $\partial X$ implies that $Av=0$ with the same boundary conditions. When $V$ is real-analytic, then the Holmgren theory for Cauchy problems can be invoked to show that $v=0$, yielding injectivity for the non homogeneous problem. 

If $A$ is injective, then so is $A+\delta V$ for $\delta V$ sufficiently small since $A^{-1}\delta V$ is of spectral radius strictly less than $1$ for $\delta V$ sufficiently small (in an appropriate topology and for an open set that depends on $V$). As a consequence, we obtain that there is a dense open set of potentials $V$ (in that same appropriate topology) such that $A+V$ is injective. This property (valid on a dense open set) is called generic injectivity.

When $V$ is not analytic, then a considerably more delicate Unique Continuation Property (UCP) for $A$ shows again that $V=0$. We do not consider UCP here and refer to \cite{B-Irvine-12} for the analysis of the case $\sigma\equiv0$.

Coming back to the UMOT problem, \eqref{eq:PQJ} says that  $A_Jv=S_J$, which we can modify as
\begin{equation}
\label{eq:normalAJ}
A^*_J A_J v = A^*_J S_J\mbox{ in } X,\qquad v=0 \qquad \partial_\nu v=g \mbox{ on } \partial X.
\end{equation}
Since we prescribe Dirichlet data in (\ref{eq:diff0}) to be the same as that of the background solutions, $v=0$ on the boundary. However, the data for $\partial_\nu u_j$ is then determined by these solutions of (\ref{eq:diff0}), and therefore $\partial_\nu \delta u_j$  is nonzero in general. Nonetheless, we do assume this is known boundary data and represent it as $g$  in (\ref{eq:normalAJ}). 
\begin{theorem}
\label{thm:holmgren}[Holmgren]
Let us assume that all coefficients $(\gamma,\sigma,\{u_j\})$ are real-analytic so that $A^*_J A_J$ is an operator with real-analytic coefficients.  Assume moreover that $P_J$ is elliptic in the sense that \eqref{eq:ellcond} holds. Then $A^*_J A_J$ and $A_J$ augmented with the boundary conditions $v=\partial_\nu v=0$ on $\partial X$ are injective operators. Moreover, the problem (\ref{eq:normalAJ}) has a unique solution satisfying 
\begin{equation} \label{eq:normalAJest}  \| v\|_{H^s(X,\mathbb{R}^{2+J})} \leq C\left(\| \delta H_j\|_{H^s(X,\mathbb{R}^J) } + \| g\|_{H^{s-3/2}(\partial X;\mathbb{R}^{2+J})}    \right).  \end{equation} 
\end{theorem}
\begin{proof}
 Injectivity is a direct consequence of Holmgren's theory for systems of operators as presented, e.g., in \cite{H-I-SP-83} (see also \cite[Thm 3.3]{B-Irvine-12}).  The function $g\in H^{s-3/2}(\partial X,\mathbb{R}^{2+J})$ can be lifted to the unique $H^s(X,\mathbb{R}^{2+J})$ function $\phi$ such that 
 \begin{equation}\label{liftphi} 
 \Delta\Delta\phi=0\ \  \mbox{in} \ \ X; \ \  \phi=0,\ \partial_\nu \phi =g\ \  \mbox{on}  \ \ \partial X. 
 \end{equation}  By ellipticity, we can apply (\ref{eq:stabest1})  to $v-\phi$ to obtain 
 \begin{equation}   \| v\|_{H^s(X,\mathbb{R}^{2+J})} \leq C\left(\| \delta H_j\|_{H^s(X,\mathbb{R}^J) } + \|\phi\|_{H^{s}( X;\mathbb{R}^{2+J})} \right) + C_2 \| v\|_{L^2(X;\mathbb{R}^{2+J})} \end{equation}   
 and injectivity precisely means that $C_2=0$ up to the choice of a possibly larger constant $C$.  We also know that
 \begin{equation}\label{sbiharmest} \| \phi\| _{H^s(X,\mathbb{R}^{J+2})}\leq C \| g\|_{H^{s-3/2}(\partial X, \mathbb{R}^{J+2})} \end{equation}
from which the result follows. 
\end{proof}

With this result, we obtain a generic injectivity statement for the linearized UMOT problem.
\begin{theorem}
\label{thm:geninj}[Generic Injectivity]
 There is a dense open set of coefficients $(\gamma,\sigma)$ (in any topology of sufficiently smooth coefficients) such that for $J\geq5$ ($J\geq3$ in dimension $n=2$) and for an open set of boundary conditions $\{f_j\}_{1\leq j\leq J}$, we have that $A^*_J A_J$ and $A_J$ augmented with the boundary conditions $v=\partial_\nu v=0$ on $\partial X$ are injective operators, and that the problem (\ref{eq:normalAJ}) has a unique solution.
\end{theorem}
\begin{proof}
 Let $(\gamma,\sigma)$ be real-analytic. Then the results of Theorems \ref{thm:cgond} and \ref{thm:ell2d} below show the existence of an open set (in the $C^1(\partial X)$ topology, say) of boundary conditions $\{f_j\}$ such that $P_J$ is elliptic. By density, we choose the boundary conditions $\{f_j\}$  to be real-analytic. As a consequence, the solutions $u_j$ are real-analytic and $P_J$ is an elliptic system of second-order operators.

 We may then apply Theorem \ref{thm:holmgren} to obtain that $A^*_J A_J$ and $A_J$ are injective. Now, injectivity of $A^*_J A_J$ and $A_J$ still holds for $(\tilde\gamma,\tilde\sigma)$ sufficiently close to $(\gamma,\sigma)$ (for instance in the $C^1(X)$ topology) and for $\tilde f_j$ close to $f_j$ (for instance in the $C^1(\partial X)$ topology).
\end{proof}

The above construction thus shows that for appropriate choices of boundary conditions, the operator $A_J$ is injective for an open, dense, set of coefficients $(\gamma,\sigma)$. Note that an open dense set may in fact be very far from accounting for all possible coefficients $(\gamma,\sigma)$. But this result still indicates that injectivity of $A_J$ and $A_J^*A_J$ does hold for a large class of coefficients.

\medskip

Another injectivity result shows that when $P_J$ is elliptic, then $A^*_J A_J$ in \eqref{eq:normalAJ} is injective on sufficiently small open domains $\Omega\ni0$. For $x\in X$, we say that $x\in\eps\Omega$ for $0<\eps\leq1$ if $\eps^{-1}x\in\Omega$. The size of the domain depends on the coefficients in $P_J$, but this results gives yet another indication that the UMOT problem is well posed in several configurations of interest.

\begin{theorem} \label{thm:small domains}[Injectivity for Small Domains] Let $\Omega$ be a subset of $X$ with $0\in\eps\Omega$ for all $0<\eps\leq\eps_0$.  Assume that the highest order part of the system $P_J$ given by (\ref{eq:PJ}) is elliptic in the sense that \eqref{eq:ellcond} holds. 
Then there exists $\epsilon_1> 0$ such that if $\epsilon<\epsilon_1 $, the  linearized system for $ v = ( \delta\gamma, \delta\sigma, \delta u_1,\ldots \delta u_J )^T$
\begin{displaymath}
 A_J^* A_J v = A_J^* S_J, \mbox{ on } \epsilon\Omega
\end{displaymath}
$$ v=0\ \ {\partial v\over{\partial \nu}}= g \mbox{ on } \ \partial(\epsilon \Omega) $$
has a unique solution $v $ satisfying 
\begin{equation}
\label{eq:regulwsmall}
  \| v \|_{H^2_0(\epsilon\Omega;\Rm^{J+2})}  \leq C_1 \|\delta H\|_{H^2(\epsilon\Omega;\Rm^{J})} + C_\epsilon  \| g \|_{H^{1/2}(\partial (\epsilon\Omega);\Rm^{J+2})}\end{equation}
  where $C_1$ is independent of $\epsilon$. 
\end{theorem}
{\it Proof. }\ \ Recall that $A_J=P_J + Q_J$ where all terms of $Q_J$ are at most first order. We let  $A_J(0)=P_J(0) +Q_J(0)$ be  the operator obtained by freezing all coefficients $(\gamma,\sigma,\{u_j\},\{\nabla u_j\})$ at $x=0$. Then $ A_J= A_J(0) + \big( A_J- A_J(0) \big)$   and we set the remainder $ T=   \big( A_J- A_J(0) \big)$. 
Let us lift $g$ to the unique $H^2$ function $\phi$ defined by (\ref{liftphi}). Then we have that
\begin{equation}\label{biharmest} \| \phi\| _{H^2(\epsilon\Omega)}\leq C_\epsilon \| g\|_{H^{1/2}(\partial(\epsilon\Omega))} \end{equation}
where $C_\epsilon$ may depend on $\epsilon$. Let us subtract off $\phi$ so that $w=v-\phi$ solves
\begin{displaymath}
 A_J^* A_J w = A_J^*( S_J + A_J\phi), \mbox{ on } \epsilon\Omega
\end{displaymath}
$$ w=0\ \ {\partial w\over{\partial \nu}}= 0 \mbox{ on } \ \partial(\epsilon \Omega). $$

Note that $P_J(0)$ also satisfies the ellipticity condition (\ref{eq:ellcond}) by assumption, and is homogeneous of degree $2$. Since the coefficients of $P_J(0)$ are all constant, it is easy to see by using Plancherel's theorem that 
\begin{equation} \label{coercive} \int_{\epsilon\Omega} \| P_J(0) w \|^2 \geq \alpha \| w \|_{H^2_0(\epsilon \Omega)} ,\end{equation} 
where $\alpha$ does not depend on $\epsilon$.  
Suppose $$A_J^*(0)A_J(0) w=B^*f$$ for $w\in H^2_0(\epsilon \Omega)$ where $B$ is a second order differential operator. Then $B$ is bounded independently of $\epsilon$  from $H^2$ to $L^2$ by $\| B\|$.   Integration by parts yields
$$ \int_{\epsilon\Omega} A_J(0) w\cdot A_J(0) w = \int_{\epsilon\Omega} f B w$$
from which we can compute
\begin{eqnarray} \int_\Omega P_J(0) w\cdot P_J(0) w &=& -2\int_\Omega Q_J(0) w \cdot P_J(0) w - \int_\Omega Q_J(0) w\cdot Q_J(0) w + \int_\Omega f Bw \\
&\leq& C \| w\|_{H^1(\epsilon\Omega)} \| w \|_{H^2(\epsilon\Omega)}  + \| B\| \| f\|_{L^2(\epsilon\Omega)} \| w\|_{H^2(\epsilon\Omega)} \label{pweq} \end{eqnarray}
where we use the full $H^1$ and $H^2$ norms here. Given the Poincar\'e inequality for the fixed domain $\Omega$ and $w\in H^2_0(\Omega)$,
$$  \|  w\|_{H^2(\Omega)} \leq D \| w \|_{H^2_0(\Omega)} ,$$
we can deduce from scaling the following Poincar\'e inequalities for the small domain $\epsilon \Omega$ 
\begin{equation}\label{poincare1}  \| w \|_{H^2(\epsilon\Omega)} \leq D \| w \|_{H^2_0(\epsilon\Omega)} ,\end{equation}
and 
\begin{equation} \label{poincare2}  \| w \|_{H^1(\epsilon\Omega)} \leq D\epsilon \| w \|_{H^2_0(\epsilon\Omega)} \end{equation}
for all $w\in H^2_0(\epsilon\Omega)$.  Applying (\ref{poincare1}) and (\ref{coercive}) on the left hand side of (\ref{pweq}) yields
$$ \alpha \| w\|^2_{H^2_0(\epsilon\Omega)} \leq {CD} \left( \| w\|_{H^1(\epsilon\Omega)}\| w\|_{H^2_0(\epsilon\Omega)} + \| B\| \| f\|_{L^2(\epsilon\Omega)} \| w\|_{H^2_0(\epsilon\Omega)}\right) . $$
Now we use (\ref{poincare2}) and assume that $\epsilon$ is small enough so that
$$ \alpha - CD^2\epsilon >0 .$$ Bringing the first term on the right hand side over to the left and dividing by $\| v\|_{H_0^2(\epsilon\Omega)}$ gives us
\begin{equation} \label{small1} \| w\|_{H^2_0(\epsilon\Omega)} \leq {CD\over{\alpha-CD^2\epsilon}} \| B\|  \| f\|_{L^2(\epsilon\Omega)}  . \end{equation}
This shows that $(A^*_J(0)A_J(0))^{-1} B^*$ (with zero boundary conditions) exists and is bounded as an operator from $L^2$ to $H^2_0$, independently of small domain size $\epsilon$.  Indeed we have
\begin{equation} \label{smalldomainest} \| (A^*_J (0)A_J(0))^{-1}B^* \| \leq {CD\over{\alpha-CD^2\epsilon}} \| B\| .  \end{equation} 
Now, for smooth enough coefficients, the remainder $T= \big( A_J- A_J(0) \big)$ will be small in norm from $H^2(\epsilon\Omega)$ to $L^2(\epsilon\Omega)$ due to the smallness of the domain. 
%
We now want to show that 
\begin{equation}\label{eq:injnormal}
  A_J^*A_J=(A_J(0) +T)^*(A_J(0) + T) w =0 ,\quad w=\partial_\nu w=0,
\end{equation}
implies that $w=0$, in other words that $(A_J(0)  +T)^*(A_J(0) + T)$ is injective on $H^2_0$.  The above implies that
\begin{eqnarray} 
  w &=&   (A^*_J(0)A_J(0))^{-1}  \big( (A_J(0) +T)^* (A_J(0) + T) - A^*_J(0)A_J(0)\big) w \\  &=&  (A^*_J(0)A_J(0))^{-1}  \big( A^*_J(0)T+ T^*A_J(0) + T^* T \big) w. \label{wfixed}
\end{eqnarray}
Now, for $\epsilon $ sufficiently small and the coefficients in the equations sufficiently smooth,  $T$ is small in norm from $H^2$ to $L^2$. Hence applying (\ref{smalldomainest}) to each of the three terms in (\ref{wfixed}) implies that $w\equiv0$.  Furthermore, the linearized inverse with the data subtracted off, $w=v-\phi$ satisfies
 $$ (A_J(0)+T)^*(A_J(0)+T)w =(A_J(0)+T)^*(S_J +A_J\phi ) , $$  from which the smallness of $T$, the fact that the components of $S$ are either zero or
 $L\delta H_j $, and the estimate (\ref{smalldomainest}) will yield 
 \begin{equation}
  \| w \|_{H^2_0(\epsilon\Omega;\Rm^{J+2})}  \leq C\left(  \|\delta H\|_{H^2(\epsilon\Omega;\Rm^{J})} +\|\phi\|_{H^2(\epsilon\Omega;\Rm^{J+2})} \right)\label{eq:regulwconst}
\end{equation}
which, using (\ref{biharmest}), implies
 (\ref{eq:regulwsmall}). $\Box$

We recall that in all the settings where $A^*_JA_J$ can be proven to be injective, then we have the ellipticity estimate \eqref{eq:stabest1} with $C_2=0$ for $w=v-\phi$
where $\phi$ solves (\ref{liftphi}).  In other words,
\begin{equation}
\label{eq:stabest3}
  \|(\delta\gamma,\delta\sigma)\|_{H^{s}(X;\Rm^{2})} + \|u_j\|_{H^{s}(X;\Rm^{J})} \leq C \left( \|\delta H_j\|_{H^{s}(X;\Rm^J)} +\|  g \|_{H^{s-3/2}(\partial X;\Rm^{J+2})} \right).
\end{equation} 
This shows that the reconstruction of $(\delta\gamma,\delta\sigma)$ from knowledge of $\delta H_j$ and $g$ is obtained with no loss of derivatives when $P_J$ is elliptic and $A^*_JA_J$ is injective.

\section{Local reconstruction of the nonlinear problem}
\label{sec:nonlinear}

The elliptic stability estimate \eqref{eq:stabest3} obtained in the preceding section when $P_J$ is elliptic and $A_J^*A_J$ is injective holds for coefficients $(\gamma,\sigma,u_j)\in \mX^{s}=H^s(X)\times H^s(X)\times H^{s+1}(X;\Rm^J)$ for $s>\frac n2$. The availability of such an estimate allows us to solve the nonlinear problem \eqref{eq:diff}-\eqref{eq:Hj} locally after a few algebraic manipulations.

Let us recast the nonlinear problem \eqref{eq:diff}-\eqref{eq:Hj} as 
\begin{equation}
\label{eq:nonlinj}
\mF_{2j-1}(v) = 0, \qquad \tilde \mF_{2j}(v) = H_j,\qquad 1\leq j\leq J.
\end{equation}
The construction of the linear operator $A_J$ in the preceding section is based on the commutation relation \eqref{eq:dHj2}. We need to construct an inverse problem for $v$ with $A_J$ as its Fr\'echet derivative. We also need to transform the above over-determined system into a determined system of equations if we want to invert it numerically. To do so, we consider a point $v_0=(\gamma,\sigma,u_j)\in\mX^{s}$ and assume that the linear differential operator $A_J$ constructed in the preceding section with coefficients given by $v_0$ is such that $A_J^*A_J$  augmented with boundary conditions as in \eqref{eq:normalAJ} is injective. Let $L$ and $M_j$ be the operators in \eqref{eq:deltau}-\eqref{eq:Mj} obtained with the coefficients $v_0$. We recast the inverse problem \eqref{eq:nonlinj} as
\begin{equation}
\label{eq:nonlinj2}
\mF_{2j-1}(v) = 0, \qquad \mF_{2j}(v) := L\tilde\mF_{2j}(v)-M_j \mF_{2j-1}(v) = LH_j := K_j,\qquad 1\leq j\leq J.
\end{equation}
The above problem is recast using the more compact notation 
\begin{equation}
\label{eq:nonlin}
\mF(v) = \mK,
\end{equation}
with $\mK:=L\mH$ the $2J$-vector of sources with $J$ of them non-vanishing. We therefore view the sources in a space identified with $\mY^{s-2}=H^{s-2}(X;\Rm^J)$. We also identify $\mH$ with the sources $\{H_j\}_{1\leq j\leq J}$ in $\mY^s=H^{s}(X;\Rm^J)$.

The above nonlinear problem is still over-determined and may not admit solutions unless $\mK$ satisfies compatibility conditions. We thus apply the adjoint operator $A_J^*$ as we did in earlier sections and denote the solution of 
\begin{equation}\label{eq:AJ0}
 A^*_JA_J w = A^*_J S \mbox{ in } X,\quad w=0 \quad \partial_\nu w=g \mbox{ on } \partial X,
\end{equation}
as $w=Q_0S + R_0 g$, with $Q_0$ mapping continuously $\mY^{s-2}$ to $\mX^s$ and $R_0$ mapping continuously $H^{s-\frac32}(\partial X;\Rm^{2+J})$ to $\mX^s$.

Two nonlinear procedures may then be set up depending on the construction of $v_0$. We can first assume that $v_0$ is such that $\mF(v_0)=\mK_0=(\{LH_{j,0}\})$. This may be constructed by using a first guess for $(\gamma,\sigma)$ and then solving for the $u_j$. We then observe that 
\begin{displaymath}
  \mF'(v_0) (v-v_0) = \mF(v)-\mF(v_0) - \big( \mF(v)-\mF(v_0)-\mF'(v_0) (v-v_0)\big) =\mK-\mK_0  + \mG(v-v_0),
\end{displaymath}
where $\mG(v-v_0)$ is quadratic in $v-v_0$ and bounded in $\mY^{s-2}$ by $C\|v-v_0\|_{\mX^s}^2$. Applying $A_J^*$, we obtain
\begin{displaymath}
  A_J^* \mF'(v_0) (v-v_0)  = A_J^*\big(\mK-\mK_0\big) + A_J^*\mG(v-v_0),
\end{displaymath}
since $\mF'(v_0) =A_J$ by construction of $\mF$ in \eqref{eq:nonlinj2}-\eqref{eq:nonlin}. 
Let us define $g=\partial_\nu v$ and $g_0=\partial_\nu v_0$ on $\partial X$. Then we find that 
\begin{equation}
\label{eq:fixedpoint}
 v-v_0 = Q_0 A_J^*\big(\mK-\mK_0\big)  + R_0(g-g_0) + (A_J^*A_J)^{-1}_D A^*_J\mG(v-v_0) = f_0 + \mJ(v-v_0),
\end{equation}
with $f_0=Q_0 A_J^*\big(\mK-\mK_0\big)  + R_0(g-g_0)$ and $\mJ(\phi)=(A_J^*A_J)^{-1}_D A^*_J\mG(\phi)$.
Here, $(A_J^*A_J)^{-1}_D$ is the inversion of the operator $A_J^*A_J$ with vanishing boundary conditions (of the form $v=0$ and $\partial_\nu v=0$ on $\partial X$).

For $f_0$ sufficiently small in $\mX^s$, which means for $\mH-\mH_0$ sufficiently small in $\mY^s$ and $g-g_0$ sufficiently  small in $H^{s-\frac32}(\partial X;\Rm^{2+J})$, we then easily verify that $\phi\mapsto f_0 + \mJ(\phi)$ is a contraction on a sufficiently small ball in $\mX^s$. This shows that the above equation \eqref{eq:fixedpoint} admits a unique solution that can be computed by the converging algorithm $v^{k+1}-v_0=f_0+ \mJ(v^k-v_0)$.

Alternatively, we may construct $v_0$ as the solution to 
\begin{displaymath}
  A_J^* \mF(v_0) = A_J^* \mK_0 \mbox{ in } X,\quad \partial_\nu v_0=g \mbox{ on } \partial X,
\end{displaymath}
where we also prescribe $v_0$ on $\partial X$. The main difference with respect to the preceding setting is that here $v_0$ is constructed as the solution of a fourth-order system with its Neumann condition may be chosen so that $\partial_\nu v_0=\partial_\nu v$, which we need to assume is known to obtain an invertible operator $A_J^*A_J$. In such a setting, we obtain
\begin{equation}
\label{eq:fixedpoint2}
 v-v_0 = Q_0 A_J^*\big(\mK-\mK_0\big)   + (A_J^*A_J)^{-1}_D A^*_J\mG(v-v_0) = \tilde f_0 + \mJ(v-v_0),
\end{equation}
with $\tilde f_0=Q_0 A_J^*\big(\mH-\mH_0\big)$ provided that $A_J^*A_J$ remains invertible.  For $\mH-\mH_0$ sufficiently small in $\mY^s$, we again obtain that the solution to the above equation is unique and computed by the converging algorithm $v^{k+1}-v_0=\tilde f_0+ \mJ(v^k-v_0)$.
%
%

In the latter setting, assume that  $v$ and $\tilde v$ are two solutions of \eqref{eq:nonlin} with right-hand side given by $\mK=L\mH$ and $\tilde \mK=L\tilde\mH$, respectively. Then we find that 
\begin{equation}
\label{eq:stabnonlin}
 \|\gamma-\tilde\gamma\|_{H^s(X)} + \|\sigma-\tilde\sigma\|_{H^s(X)} \leq C \| \mH-\tilde\mH\|_{H^s(X;\Rm^J)}.
\end{equation}
In other words, the iterative algorithm \eqref{eq:fixedpoint2} converges to a solution that satisfies the optimal (elliptic) stability estimate \eqref{eq:stabnonlin}. This result is summarized as:
\begin{theorem}
\label{thm:nonlinear}
  Let $v_0=(\sigma_0,\gamma_0,\{u_{j0}\})\in\mX^s$ be a point such that $P_J$ is elliptic and $A_J^*A_J$ is injective. Let  $\mH_0$ be given by \eqref{eq:nonlin}.
  Let $v=(\sigma,\gamma,\{u_{j}\})\in\mX$ and $\tilde v=(\tilde\sigma,\tilde\gamma,\{\tilde u_{j}\})\in\mX$ and let $\mH$ and $\tilde \mH$ be given in \eqref{eq:nonlin} by solving the problems \eqref{eq:diff}-\eqref{eq:Hj}. Then if $\mH$ and $\tilde\mH$ are sufficiently close to $\mH_0$ in $\mY$, then the error estimate \eqref{eq:stabnonlin} holds.
\end{theorem}

The above theorem, which is the main contribution of this paper, shows that the UMOT problem is {\em well-posed} in the sense that small errors in the available internal functionals (in a norm $H^s(X)$ with $s$ sufficiently large) translate into small errors in the reconstruction of $(\gamma,\sigma)$ in the same norm.  Moreover, the iterative algorithm based on \eqref{eq:fixedpoint2} converges to the unique solution of the UMOT problem. In the setting of \eqref{eq:fixedpoint}, we also obtain a converging algorithm with a similar estimate as \eqref{eq:stabnonlin}  with an additional term of the form $\|g-\tilde g\|_{H^{s-\frac32}}(\partial X;\Rm^{J+2})$ on the right-hand-side.
As in the analysis of many hybrid inverse problems, our proof for such results requires that the coefficients one wishes to reconstruct be sufficiently regular \cite{B-IO-12}.

\section{Conditions for ellipticity.}
\label{sec:cond}
In this section, we collect several sufficient conditions that ensure the ellipticity of the operator $P_J$. We show that in the general case, there exist sufficient background solutions so that the linearized operator $P_J$ is elliptic. Recall that $$ \theta_j=\dfrac{\nabla u_j}{|\nabla u_j|},\qquad d_j=\dfrac{u_j}{|\nabla u_j|}, $$ and the system is elliptic if 
\begin{equation} \label{determinants}
  (|\theta_i|^2 - 2(\theta_i\cdot\hat\xi )^2 )d_j^2  =  (|\theta_j|^2 - 2(\theta_j\cdot\hat\xi )^2) d_i^2   \qquad \forall i,j=1\ldots J 
\end{equation}
has no solutions  $\hat\xi = \frac{\xi}{|\xi |}\in\mathbb{S}^{n-1}$. 

In dimension $n\geq3$, the following result says that locally four solutions are sufficient to guarantee ellipticity. This number climbs to five if we want to guarantee ellipticity on a domain of arbitrary size. The proof  is based on the construction of complex geometric optics (CGO) solutions, see \eqref{eq:cgo} below. The use of CGO solutions to solve such problems was introduced in \cite{BU-IP-10}; see \cite{B-IO-12} for a review of the method. 
\begin{theorem}
\label{thm:cgond} Let $X\subset\Rm^n$ a bounded open domain. Then there is an open set of boundary conditions $\{f_j\}$ for $1\leq j\leq 5$ such that the operator $P_J$ is elliptic.
\end{theorem} 
\begin{proof}
  Consider solutions of \eqref{eq:diff} of the form
  \begin{equation}
\label{eq:cgo}
  u_\rho (x) ={1\over{\sqrt{\gamma}} }e^{\rho\cdot x} (1 + \psi_\rho (x)),
\end{equation}
where $\rho\in\Cm^n$ is a complex vector such that $\rho\cdot\rho=0$. Note that the real and imaginary parts of $u_\rho$ above are then also solutions of the equation  \eqref{eq:diff} with real-valued coefficients.

It is shown in \cite{BU-IP-10} that $\psi_\rho$ is of order $|\rho|^{-1}$, e.g., in the $C^{1}(X)$ norm provided that the coefficients $(\gamma,\sigma)$ are sufficiently smooth. As a consequence, all properties that can be shown when $\psi_\rho$ is set to $0$, such as the absence of solutions to \eqref{determinants}, extend to solutions of the form \eqref{eq:cgo}.  We will therefore find five real-valued solutions of the form \eqref{eq:cgo} such that \eqref{determinants} admits no solution. By continuity, any choice of $f_j$ sufficiently close (in any sufficiently strong topology, say $C^1(\partial X)$) to the traces of said solution on $\partial X$,  will generate solutions such that  \eqref{determinants} admits no solution.

It thus remains to prove that we can find five (locally four) solutions of the form $\Re(e^{i\rho\cdot x})$ or $\Im(e^{i\rho\cdot x})$ such that  \eqref{determinants} admits no solution (we easily verify that for $\rho$ sufficiently large, the presence of the term $\sqrt\gamma$ does not modify our conclusions, see \cite{B-IO-12}).

Let $M$ be a large real number and $\rho=|k|(e_1+ie_2)$. We denote by $u_1$ and $u_2$ the real and imaginary parts of $u_{M\rho}$, that is
\begin{displaymath}
  u_1 = e^{M|k|x_1}\cos \, M|k|x_2,\qquad u_2 = e^{M|k|x_1}\sin \, M|k|x_2
\end{displaymath}
We find
\begin{displaymath}
\theta_1 = \left( \begin{array}{c} \cos{M|k| x_2} \\ -\sin{M|k| x_2} \end{array} \right) , \quad \theta_2 = \left( \begin{array}{c} \sin{M|k| x_2} \\ \cos{M|k| x_2}\end{array} \right) , \qquad d_1 = {\cos{M|k|x_2}\over{M|k|}}, \quad d_2 = {\sin{M|k|x_2}\over{M|k|}}
\end{displaymath}
We also define $u_3=u_1+u_2$, which is also clearly a harmonic solution.  Finally, we define $u_4$ and $u_5$ exactly as $u_1$ and $u_2$ above with $M$ replaced by $1$. For $1\leq j\leq 3$ and $4\leq m\leq 5$, we have from \eqref{determinants}
\begin{displaymath}
   \big(1-2(\theta_j\cdot\hat\xi)^2\big) d_m^2 = \big(1-2(\theta_m\cdot\hat\xi)^2\big) d_j^2.
\end{displaymath}
Since $d_4^2+d_5^2=|k|^{-2}$, at least one of them does not vanish (this is why we need five solutions whereas four solutions suffice locally with either $d_4$ or $d_5$ bounded away from $0$). For $M$ sufficiently large, we thus find that  $\big(1-2(\theta_j\cdot\hat\xi)^2\big)$ is as small as we please for $1\leq j\leq 3$. But it is proved in \cite{B-Irvine-12} that this is not possible for any $\hat \xi\in\Sm^{n-1}$. In other words, the latter results states that when $\sigma=0$, then three solutions such as $u_1$, $u_2$, and $u_3$ above are sufficient to provide ellipticity. This concludes the proof of the theorem.
\end{proof}

In dimension $n=2$, the above proof combined with the results in \cite{B-Irvine-12} may be used to show that four solutions $(J=4)$ guarantee ellipticity. However, in fact three well chosen solutions (via well-chosen boundary conditions $f_j$) guarantee ellipticity as the following result demonstrates.
\begin{theorem}
\label{thm:ell2d} In dimension $n=2$, the results of Theorem \ref{thm:cgond} hold for $J=3$.
\end{theorem}
\begin{proof}
 Consider the same solutions $u_j$ for $1\leq j\leq 3$ as in the preceding theorem with $M=1$. Knowledge of $H_j$ for $1\leq j\leq 3$ is equivalent to knowledge of 
 \begin{displaymath}
H_{12}=\gamma\nabla u_1\cdot\nabla u_2 + \eta\sigma u_1 u_2 .
\end{displaymath}
This is clear by polarization and is a consequence of the fact that our internal functionals are quadratic in the solutions $u_j$.

We then calculate that the operator $P_J$  is elliptic, or equivalently \eqref{determinants} admits no solution, when the system 
  $$(|\theta_1|^2 - 2(\theta_1\cdot\hat\xi )^2 ) d_2^2 =   (|\theta_2|^2 - 2(\theta_2\cdot\hat\xi )^2 ) d_1^2$$
  $$  - 2\theta_1\cdot\hat\xi \theta_2\cdot\hat \xi d_1^2 =(|\theta_1|^2 - 2(\theta_1\cdot\hat\xi )^2 ) d_1d_2 $$
  $$  - 2\theta_1\cdot\hat\xi \theta_2\cdot\hat \xi d_2^2 =(|\theta_2|^2 - 2(\theta_2\cdot\hat\xi )^2 )  d_1d_2 $$
has no solutions for unit $\hat{\xi}\in\Sm^1$. Note that since $\theta_1$ and $\theta_2$ are unit and orthogonal,
$$ \theta_1\cdot\hat{\xi}= \cos{\alpha}, \ \ \theta_2\cdot\hat{\xi}=\sin{\alpha} $$
for $\alpha$ the angle between $\theta_1$ and $\xi$.  
Let us assume that all the above equations hold. Plugging this into the above system  means that
 \begin{eqnarray}  (1-2\cos^2{\alpha})d_2^2 &=& (1-2\sin^2{\alpha})d_1^2 \label{det1} \\ -2\cos{\alpha}\sin{\alpha}d_1^2 &=& (1-2\cos^2{\alpha})d_1d_2 \label{det2} \\   -2\cos{\alpha}\sin{\alpha}d_2^2 &=& (1-2\sin^2{\alpha})d_1d_2 . \label{det3} \end{eqnarray} 
 Equation (\ref{det1}) is recast as 
 $$ (\sin^2{\alpha}-\cos^2{\alpha} )d_2^2=- (\sin^2{\alpha}-\cos^2{\alpha} )d_1^2 $$
 which implies that 
 $$\sin^2{\alpha}-\cos^2{\alpha}=0, $$ 
 since $d_1^2,d_2^2$ are nonnegative and not both zero. Now, either (\ref{det2}) or (\ref{det3}) implies that
 $$\sin{\alpha}\cos{\alpha}=0$$
 which is a contradiction since we must have $\sin{\alpha}=\pm \cos{\alpha}$. So  the symbol of the system is elliptic. The rest of the proof is the same as that of Theorem \ref{thm:cgond}.
\end{proof}

\section{Constant Background.}
\label{sec:constant}
For the important case when the background $\gamma,\sigma$ are constant, we can actually back-substitute for $\delta u_j$ in (\ref{eq:dHj}), (\ref{eq:deltau}) and still obtain a system of differential equations if one chooses the boundary data correctly. We this approach, we still get full ellipticity, injectivity and stability for  the linearized problem, which here becomes an explicit  $2\times 2$ fourth order elliptic system of PDEs with constant coefficients.  Let us take special background real CGO-like solutions $i=1\ldots J$  of
the form $$u_i=e^{\sqrt{\frac{\sigma}{\gamma}} x\cdot v_i }$$ where the $v_j$ are vectors of unit length.  With these (exact) solutions we
have $$\theta_i=\frac{\nabla u_i}{|\nabla u_i| }= v_i$$
and $$d_i =\frac{u_i}{|\nabla u_i| }=\sqrt{\frac{\gamma}{\sigma}}.$$
Again we take data of the form $H_i$ (\ref{eq:Hj}). Then all determinants of the system (\ref{eq:hpJ})  vanish when 
$$(|\theta_i|^2 -2(\theta_i\cdot\hat{\xi} )^2 ) d_j^2= (|\theta_j|^2 -2(\theta_j\cdot\hat{\xi} )^2 ) d_i^2$$
for all $i,j=1\ldots J$. Since we always have here that $d_i^2=d_j^2$ and $|\theta_i|^2=1$, this is equivalent to
\begin{equation} \label{squareeqs}  (\hat{\xi} \cdot v_i)^2=(\hat{\xi} \cdot v_j)^2 \end{equation}
for all $i,j$. So, we have ellipticity if we take enough $v_j$ so that there is no unit length $\hat{\xi}$ for which this holds. When $n=2$, it suffices to take the three background solutions with
$$v_1=e_1, v_2=e_2, v_3=(\sqrt{2}/2,\sqrt{2}/2 ). $$
When the dimension $n=3$, we can take the $4$ directions
\begin{eqnarray*} v_j=e_j,\quad 1\leq j\leq 3,\qquad  v_4 &=& (1/\sqrt{3} ,1/\sqrt{3}, 1/\sqrt{3} ) .  \end{eqnarray*}
For any dimension we can find such directions to guarantee ellipticity if we allow ourselves $n+1$ background solutions (which is optimal in dimensions $n=2$ and $n=3$ to obtain ellipticity as demonstrated earlier). We can, for example take 
 \begin{eqnarray*} v_j=e_j,\quad 1\leq j\leq n,\qquad  v_{n+1} = (0 , w_1,\ldots, w_{n-1} ) ,  \end{eqnarray*}
 where $w\in\mathbb{R}^{n-1}$ is any unit vector for which no possible sums of the components of the form
 $$ \pm w_1\pm w_2 \pm\ldots \pm w_{n-1} =1. $$ 

Now, we know from using Holmgren's theorem in Section \ref{sec:injectivity}, the linearized problem here is invertible using boundary data described above. However, in this section we process the data slightly differently. Note that the ellipticity conditions don't change when we eliminate the unknowns $\delta u_j$ from the system  (\ref{eq:dHj}),(\ref{eq:deltau}). 
From (\ref{eq:dHj}),
\begin{equation}
\delta H_{i} (\delta\gamma,\delta\sigma) = \delta \gamma | \nabla u_i |^2 +2 \gamma  \nabla\delta u_i\cdot\nabla u_i +\eta\delta\sigma (u_i)^2 
+2\eta\sigma  \delta u_i u_i. \end{equation}
Using $$\nabla u_i= \sqrt{{\sigma\over{\gamma}} }u_i v_i,$$ 
we have that 
\begin{equation}
{\delta H_{i}\over{u_i} }  = \delta \gamma {\sigma \over{\gamma }} u_i+2 \gamma \sqrt{\frac{\sigma }{\gamma }} \nabla\delta u_i\cdot v_i+ \eta\delta\sigma u_i + 2\eta\sigma  \delta u_i. \end{equation}
Now, to eliminate $\delta u_j $  we apply $L$ to both sides:
\begin{multline}
L {\delta H_{i}\over{u_i} }  = L( \delta \gamma ) {\sigma \over{\gamma }} u_i -  {\sigma ^2\over{\gamma }}\delta\gamma u_i -\sigma \nabla\delta\gamma\cdot v_i u_i +2 \gamma \sqrt{\frac{\sigma }{\gamma }}L(  \nabla\delta u_i\cdot v_i) \\ + \eta L( \delta\sigma)  u_i -\sigma \eta\delta\sigma u_i -\gamma \eta\nabla \delta\sigma\cdot v_i u_i + 2\eta\sigma  L (\delta u_i ), \end{multline} and we can calculate 
$$ L (\delta u_i) = \nabla \delta\gamma\cdot v_i \sqrt{\frac{\sigma }{\gamma }} u_i +\delta\gamma{\sigma \over{\gamma }} u_i -\delta\sigma u_i $$
which yields
\begin{eqnarray} L(\nabla \delta u_i\cdot v_i  ) & = & \nabla L(\delta u_i) \cdot v_i \nonumber \\
&=&\nabla  ( \nabla \delta\gamma\cdot v_i)\cdot v_i\sqrt{\frac{\sigma }{\gamma }} u_i+ 2 \nabla \delta\gamma\cdot v_i {\sigma \over{\gamma  }} u_i +
\delta\gamma {\sigma \over{\gamma  }}\sqrt{\frac{\sigma }{\gamma }} u_i -\delta\sigma \sqrt{\frac{\sigma }{\gamma }} u_i - \nabla\delta\sigma\cdot v_i u_i.  \nonumber \end{eqnarray}
We obtain
\begin{equation}
{1\over{ \sigma  u_i} } L {\delta H_{i}\over{u_i} }  = C_{i} \delta \gamma + B_{i} \delta\sigma \nonumber \end{equation}
where $C_i$ and $B_i$  are the second order constant coefficient differential operators
$$C_i = -\Delta +2\partial_{v_i}^2 + (3+2\eta)\sqrt{\frac{\sigma }{\gamma }}\partial_{v_i} + 2(1+\eta){\sigma \over{\gamma }} $$
and
$$B_i = -\eta{\gamma \over{\sigma }}\Delta -(2+\eta)\sqrt{\frac{\gamma }{\sigma }}\partial_{v_i} -2(1+\eta).$$
Define the data as
\begin{equation} \label{eq:constS} \tilde{S}_i= {\tilde L} \delta H_i = \left({1\over{ \sigma  u_i} } L{1\over{u_i} }\right) \delta H_i \end{equation}
and define the operator
\begin{equation} {\label{eq:constgamma} }
 \tilde{A}_J =  \left(\begin{matrix} C_{i} & B_{i}  \end{matrix}\right)
\end{equation}
with as many rows as we have $i$'s.  This is a constant coefficient elliptic PDE system, but due to the possible presence of eigenvalues, Dirichlet conditions are not obviously enough to guarantee injectivity. To resolve this, we again apply the technique described in Section \ref{sec:injectivity}. That is, we make the problem square by applying the adjoint. In this case, however, the result is a $2\times2$ fourth order system of differential equations
\begin{displaymath}
\tilde{A}_J^* \tilde{A}_J \left(\begin{matrix} \delta \gamma \\ \delta\sigma  \end{matrix}\right) = \tilde{A}_J^* \tilde{S}.
\end{displaymath}
More precisely, we have
\begin{displaymath}
\tilde{A}_J^*\tilde{A}_J = \left(\begin{matrix} \dsum_{i} C^*_{i}C_{i} & \dsum C^*_{i}B_{i}  \\ \dsum_{i} B^*_{i} C_{i} & \dsum_{i} B^*_{i} B_{i} \end{matrix}\right).
\end{displaymath}
Since $\tilde{A}_J^*\tilde{A}_J$ is elliptic in the sense of Douglis-Nirenberg, the system will admit a parametrix if we apply the Dirichlet conditions $w=0$ and $\partial_\nu w=0$ on $\partial X$ . Again by using Holmgren, we have that the operator $\tilde{A}_J^*\tilde{A}_J$ is actually injective with these conditions. Indeed, let us assume that $\tilde{A}_J^*\tilde{A}_J w=0$ in $X$ and $w=0$ and $\partial_\nu w=0$ on $\partial X$ . Then integrating by parts, we obtain that $\tilde{A}_J w=0$ in $X$. Since $\tilde{A}_J $ is elliptic, this implies that $\partial_\nu^2 w=0$ so that all second-order derivatives of $w$ vanish. Upon differentiating the equation for $w$, we obtain that all derivatives of order less than or equal to $3$ vanish on $\partial X$ (as well as higher-order derivatives in fact). From Holmgren we deduce that $w\equiv0$ in the vicinity of $\partial X$ and by induction in the whole domain $X$. So, the leading order operators are invertible stably for constant ($\gamma $, $\sigma $), and by continuity for background coefficients which are near to constant. We have proven the following theorem.
\begin{theorem}  Suppose the background coefficients  $(\gamma ,\sigma )$ are constant and the domain $X\subset\mathbb{R}^n$. Then if one applies data $f_i$ corresponding to  $n+1$ background solutions  $$ u_i=e^{\sqrt{\frac{\sigma }{\gamma }} x\cdot v_i }$$ for the $\{ v_i\}$ any set of unit vectors for which 
\begin{equation} \label{squareeqs2}  (\hat{\xi} \cdot v_i)^2=(\hat{\xi} \cdot v_j)^2 \quad \forall (i,j)  \end{equation}
has no solutions $\hat{\xi}\in \mathbb{S}^{n-1}$, then the linearized system for $ w=(\delta\gamma, \ \delta\sigma )^t $ 
$$\tilde{A}^*_J \tilde{A}_J w = \tilde{A}_J^*\tilde{S} \ \mbox{in} \ \ X$$
$$ w=0\ \ {\partial w\over{\partial \nu}}= 0 \mbox{ on } \ \partial X $$
for $\tilde{A}_J $ and $\tilde{S}$ given by (\ref{eq:constgamma}), (\ref{eq:constS}) is elliptic and has a unique solution $w$ satisfying 
\begin{equation}
\label{eq:regulwconst}
  \|w\|_{H^{s}(X;\Rm^2)}  \leq C  \|dH\|_{H^{s}(X;\Rm^J)}  . 
\end{equation}
\end{theorem}
\noindent
{\bf Remark.} Note that in the above analysis we did not explicitly cover when $\sigma =0$. This 
case is actually simpler than the above, and illuminates very clearly why more than two solutions are necessary for ellipticity. 
We can of course choose the background solutions $u_0=1$ and $u_j=x_j$ for $1\leq j\leq n$. We then obtain that 
\begin{displaymath}
  dH_{0} = \eta \delta \sigma.
\end{displaymath}
This is easy enough to invert for $\delta\sigma$, which we eliminate from the system of equations and then focus on $\delta\gamma$. After elimination, we  find that
\begin{displaymath}
   \Delta d\tilde H_{i} \delta\gamma = \big(\Delta - 2\partial^2_{x_i})\delta\gamma
\end{displaymath}
which is a hyperbolic operator (in fact it is precisely the wave operator where $x_i$ acts as time). So, if we were only to use one $u_i$,  (in addition to $u_0$), it would not be possible to invert for $\delta\gamma$ unless the geometry were quite specific.  Additionally, stability is lost since we have taken the Laplacian of the data. However, we will see that by taking more background solutions we regain ellipticity and stability. 
If $n\geq 3$, consider the $n$ background solutions $\{ u_i \} $ together.  We in fact have that
\begin{displaymath}
\dsum_{i=1}^n \Delta d\tilde H_{i} = (n-2)\Delta,
\end{displaymath}
which is clearly elliptic, and can be inverted to find $\delta\gamma$ if we impose Dirichlet boundary conditions. Therefore, we can obtain the stable reconstruction of $(\delta\gamma,\delta\sigma)$ by inverting a second order system of equations. 

In dimension $n=2$, the measurements $H_{1}$ and $H_{2}$ alone are not quite enough since $$\Delta d\tilde H_{1}=-\Delta d\tilde H_{2}.$$ We need the contribution 
$$H_{12}= \nabla u_1\cdot \nabla u_2 $$
which can be obtained by the background solution $ u_3= u_1+u_2$, and yields 
$$ \Delta d\tilde H_{12} = -2\partial_{x_1 x_2}.$$ 
We observe that we now want to solve the $2\times 1$ system  $$\tilde{A}_J \delta\gamma := \left( \begin{array}{c}   \partial^2_{x_2} - \partial^2_{x_1}\\ -2\partial_{x_1 x_2} \end{array} \right) \delta\gamma = \left( \begin{array}{c} \Delta d\tilde H_{1} \delta\gamma \\
\Delta d\tilde H_{12} \delta\gamma \end{array} \right)  . $$
One can approach this in the same manner as above and apply $\tilde{A}_J^*$ to both sides. This yields the simple bi-Laplace equation
$$\Delta \Delta \delta\gamma = \tilde{A}_J^*\left( \begin{array}{c} \Delta d\tilde H_{11} \delta\gamma \\
\Delta d\tilde H_{12} \delta\gamma \end{array} \right), $$
which we can clearly invert stably if we impose Cauchy data $\delta\gamma=0, {\partial \delta\gamma\over{\partial \nu}}=0$ on $\partial X$. Again we have converted the system to one that is square with the cost of having to go to fourth order equations and requiring more boundary data. Note that by continuity, the leading operators will also be invertible for 
$(\gamma ,\sigma  )$ in the vicinity of $( 1, 0)$ (and indeed in the vicinity of any positive constants).

%

\section*{Acknowledgments}
The authors would like to thank John Schotland for many stimulating discussions on the problem of ultrasound modulated optical tomography.


\begin{thebibliography}{10}

\bibitem{ADN-CPAM-64}
{\sc S.~Agmon, A.~Douglis, and L.~Nirenberg}, {\em {Estimates near the boundary
  for solutions of elliptic partial differential equations satisfying general
  boundary conditions. II}}, Comm. Pure Appl. Math., 17 (1964), pp.~35--92.

\bibitem{A-Sp-08}
{\sc H.~Ammari}, {\em An Introduction to Mathematics of Emerging Biomedical
  Imaging}, vol.~62 of Mathematics and Applications, Springer, New York, 2008.

\bibitem{ABCTF-SIAP-08}
{\sc H.~Ammari, E.~Bonnetier, Y.~Capdeboscq, M.~Tanter, and M.~Fink}, {\em
  Electrical impedance tomography by elastic deformation}, SIAM J. Appl. Math.,
  68 (2008), pp.~1557--1573.

\bibitem{ACDRT-SIAP-11}
{\sc H.~Ammari, Y.~Capdeboscq, F.~De~Gournay, A.~Rozanova-Pierrat, and
  F.~Triki}, {\em Microwave imaging by elastic perturbation}, SIAM J. Appl.
  Math., 71 (2011), pp.~2112--2130.

\bibitem{AS-IP-12}
{\sc S.~R. Arridge and O.~Scherzer}, {\em Imaging from coupled physics},
  Inverse Problems, 28 (2012), p.~080201.

\bibitem{AS-IP-10}
{\sc S.~R. Arridge and J.~C. Schotland}, {\em Optical tomography: forward and
  inverse problems}, Inverse Problems, 25 (2010), p.~123010.

\bibitem{AFRBG-OL-05}
{\sc M.~Atlan, B.~C. Forget, F.~Ramaz, A.~C. Boccara, and M.~Gross}, {\em
  Pulsed acousto-optic imaging in dynamic scattering media with heterodyne
  parallel speckle detection}, Optics Letters, 30(11) (2005), pp.~1360--1362.

\bibitem{B-Irvine-12}
{\sc G.~Bal}, {\em {Hybrid Inverse Problems and Systems of Partial Differential
  Equations}}, {arXiv:1210.0265}.

\bibitem{B-IP-09}
\leavevmode\vrule height 2pt depth -1.6pt width 23pt, {\em Inverse transport
  theory and applications}, Inverse Problems, 25 (2009), p.~053001.

\bibitem{B-IO-12}
\leavevmode\vrule height 2pt depth -1.6pt width 23pt, {\em {Hybrid inverse
  problems and internal functionals}}, Inside Out, Cambridge University Press,
  Cambridge, UK, G. Uhlmann, Editor, 2012.

\bibitem{B-UMEIT-12}
\leavevmode\vrule height 2pt depth -1.6pt width 23pt, {\em {Cauchy problem for
  Ultrasound modulated EIT}}, To appear in Anal. PDE. {arXiv:1201.0972v1},
  (2013).

\bibitem{BBMT-12}
{\sc G.~Bal, E.~Bonnetier, F.~Monard, and F.~Triki}, {\em Inverse diffusion
  from knowledge of power densities}, to appear in Inverse Problems and Imaging
  {arXiv:1110.4577},  (2012).

\bibitem{BNSS-JIIP-13}
{\sc G.~Bal, W.~Naetar, O.~Scherzer, and J.~Schotland}, {\em Numerical
  inversion of the power density operator}, To appear in J. Ill-posed Inverse
  Problems,  (2013).

\bibitem{BS-PRL-10}
{\sc G.~Bal and J.~C. Schotland}, {\em {Inverse Scattering and Acousto-Optics
  Imaging}}, Phys. Rev. Letters, 104 (2010), p.~043902.

\bibitem{BU-IP-10}
{\sc G.~Bal and G.~Uhlmann}, {\em Inverse diffusion theory for photoacoustics},
  Inverse Problems, 26(8) (2010), p.~085010.

\bibitem{CFGK-SJIS-09}
{\sc Y.~Capdeboscq, J.~Fehrenbach, F.~{de Gournay}, and O.~Kavian}, {\em
  Imaging by modification: numerical reconstruction of local conductivities
  from corresponding power density measurements}, SIAM J. Imaging Sciences, 2
  (2009), pp.~1003--1030.

\bibitem{DN-CPAM-55}
{\sc A.~Douglis and L.~Nirenberg}, {\em Interior estimates for elliptic systems
  of partial differential equations}, Comm. Pure Appl. Math., 8 (1955),
  pp.~503--538.

\bibitem{GS-SIAP-09}
{\sc B.~Gebauer and O.~Scherzer}, {\em Impedance-acoustic tomography}, SIAM J.
  Applied Math., 69(2) (2009), pp.~565--576.

\bibitem{GS-CUP-94}
{\sc A.~Grigis and j.~Sj{\"o}strand}, {\em Microlocal Analysis for Differential
  Operators: An Introduction}, Cambridge University Press, 1994.

\bibitem{H-I-SP-83}
{\sc L.~V. H{\"o}rmander}, {\em {The Analysis of Linear Partial Differential
  Operators I: Distribution Theory and Fourier Analysis}}, Springer Verlag,
  1983.

\bibitem{KLZZ-JOSA-97}
{\sc M.~Kempe, M.~Larionov, D.~Zaslavsky, and A.~Z. Genack}, {\em Acousto-optic
  tomography with multiply scattered light}, J. Opt. Soc. Am. A, 14(5) (1997),
  pp.~1151--1158.

\bibitem{KK-AET-11}
{\sc P.~Kuchment and L.~Kunyansky}, {\em {2D and 3D reconstructions in
  acousto-electric tomography}}, Inverse Problems, 27 (2011), p.~055013.

\bibitem{KS-IP-12}
{\sc P.~Kuchment and D.~Steinhauer}, {\em Stabilizing inverse problems by
  internal data}, Inverse Problems, 28 (2012), p.~084007.

\bibitem{MB-aniso-13}
{\sc F.~Monard and G.~Bal}, {\em Inverse anisotropic conductivity from power
  density measurements in dimensions $n\geq3$}, submitted.

\bibitem{MB-IP-12}
\leavevmode\vrule height 2pt depth -1.6pt width 23pt, {\em Inverse anisotropic
  diffusion from power density measurements in two dimensions}, Inverse
  Problems, 28 (2012), p.~084001.

\bibitem{MB-IPI-12}
\leavevmode\vrule height 2pt depth -1.6pt width 23pt, {\em Inverse diffusion
  problem with redundant internal information}, Inverse Problems and Imaging,
  6(2) (2012), pp.~289--313.

\bibitem{PS-IP-07}
{\sc S.~Patch and O.~Scherzer}, {\em Photo- and thermo- acoustic imaging},
  Inverse Problems, 23 (2007), pp.~S1--10.

\bibitem{S-SP-2011}
{\sc O.~Scherzer}, {\em Handbook of Mathematical Methods in Imaging}, Springer
  Verlag, New York, 2011.

\bibitem{S-JSM-73}
{\sc V.~A. Solonnikov}, {\em Overdetermined elliptic boundary-value problems},
  J. Sov. Math., 1 (1973), pp.~477--512.

\bibitem{W-JDM-04}
{\sc L.~V. Wang}, {\em Ultrasound-mediated biophotonic imaging: a review of
  acousto-optical tomography and photo-acoustic tomography}, Journal of Disease
  Markers, 19 (2004), pp.~123--138.

\end{thebibliography}
\end{document}